\documentclass{article}
\usepackage[utf8]{inputenc}
\usepackage{amsmath,amsthm,amssymb,amsfonts}
\usepackage[usenames]{color}

\usepackage{graphicx}
\usepackage{todonotes}
\usepackage{tikz} 
\usepackage{comment}

\newtheorem{thm}{Theorem}

\newtheorem{lem}[thm]{Lemma}

\newcommand{\N}{{\mathbb N}}

\def\1{{\bf 1}}
\def\N{\mathbb N}

\title{Does genetic diversity help survival?
}
\begin{document}

\maketitle

Luiz Renato Fontes$^1$, Fabio P. Machado \footnote{Instituto de Matemática e Estatística, Universidade de São Paulo, Rua do Matão 1010, 05508-090 São Paulo SP Brasil
E-mail: lrfontes@usp.br and fmachado@ime.usp.br} and Rinaldo B. Schinazi\footnote{Department of Mathematics, University of Colorado, Colorado Springs, CO 80933-7150, USA.
E-mail: rinaldo.schinazi@uccs.edu}

\begin{abstract}
    We introduce the following model for the evolution of a population. At every discrete time $j\geq 0$ exactly one individual is introduced in the population and is assigned a death probability $c_j$ sampled from $C$, a fixed probability distribution. We think of $c_j$ as a genetic marker of this individual. At every time $n\geq 1$ every individual in the population dies or not independently of each other with its corresponding death probability $c_j$. We show that the population size goes to infinity if and only if $E(1/C)=\infty$. This is in sharp contrast with the model with constant $c$ and with the model in random environment (same random $c_n$ for all individuals at time $n$). Both of these models are always positive recurrent. Thus, genetic diversity does seem to help survival! We also study the point process associated with our model. We show that the limit point process has an accumulation point near 0 for the $c'$s. For certain $C$ distributions, including the uniform, the limit process properly rescaled is also shown to converge to a non-homogeneous Poisson process.
\end{abstract}
\section{The model}

Consider a population with the following dynamics. At every discrete time $j\geq 0$ exactly one individual is introduced in a population and to this individual is assigned a death probability $c_j$, sampled from a random variable $C$, with support on $(0, 1)$. We think of $c_j$ as a genetic marker for that individual. At the same time that this individual is introduced, each individual present, except the one who has just entered, dies, regardless of all the others, with the death probability assigned to it at the time it was introduced into the population.

Let $i\geq 0$, we denote by $G_i$ the time that individual $i$ (i.e. the individual introduced at time $i$) survives. With our assumptions we see that $G_0,G_1,\dots,G_n$ are independent random variables with the following property. For all $i\geq 0$, the conditional distribution of $G_i$ given $c_i$ is a geometric random variable with parameter $c_i$.

\section{The population size}

For $n\geq 0$, let
$$A_n=\{(k,c_k):0\leq k\leq n\mbox{ and } G_k\geq n-k\}.$$
That is, the random set of points $A_n$ represents all the individuals alive at time $n$, starting with one individual at time $0$. 

A simple but critical observation is that $A_n$ has the same probability distribution as 
$$\Tilde A_n=\{(k,c_k):0\leq k\leq n\mbox{ and } G_k\geq k\}.$$
Denote by $|B|$ the cardinal of a set $B$.

\begin{thm} Assume that  $c_0,c_1,\dots$ are sampled from the same distribution as a fixed random variable $C$. 
\begin{enumerate}
    \item If $E(1/C)<+\infty$ then $(|A_n|)$ converges in distribution to a random variable which is a.s. finite.
    \item If $E(1/C)=+\infty$ then $\lim_{n\to\infty}|A_n|=+\infty$ in probability.
\end{enumerate}
\end{thm}
\begin{proof}
We will show that the results hold almost surely for $(\tilde A_n)$. This will imply the results for $(A_n)$.
Observe that,
$$|\tilde A_n|=\sum_{k=0}^n {\bf 1}_{\{G_k\geq k\}}.$$
Since $G_0$ is a geometric random variable with support on the natural numbers, $G_0$ is larger than 0. Hence, 
$$|\tilde A_n|=1+\sum_{k=1}^{n} {\bf 1}_{\{G_k\geq k\}}.$$
Then,
\begin{align*}
E\left(|\tilde A_n|\right)=&1+\sum_{k=1}^{n} P(\{G_k\geq k\})\\
  =&1+\sum_{k=1}^{n} E\left((1-C)^{k-1}\right)\\
  =&1+E\left(\frac{1-(1-C)^{n}}{C}\right)\\
\end{align*}

By the Monotone Convergence Theorem,
$$\lim_{n\to\infty}E\left(|\tilde A_n|\right)=1+E(\frac{1}{C}).$$

Since the sequence $(|\tilde A_n|)_{n\geq 0}$ is increasing,
$$\lim_{n\to\infty}E\left(|\tilde A_n|\right)=E\left(\lim_{n\to\infty}|\tilde A_n|\right).$$

Thus, if $E(\frac{1}{C})<\infty$ then $\lim_{n\to\infty}|\tilde A_n|<+\infty$ a.s.
Therefore, $(|\tilde A_n|)$ converges in distribution to a random variable which is a.s. finite. Since $|A_n|$ has the same distribution of $|\tilde A_n|$ for every $n\geq 0$ the proof of the first part of the theorem is complete.

We now turn to the case $E(\frac{1}{C})=\infty$. Then,
$$\sum_{k=0}^{\infty} P(G_k\geq k)=+\infty.$$
Since the random variables $G_0,G_1,\dots$ are independent the second Borel-Cantelli Lemma applies. Therefore, a.s. there are infinitely many $k$'s such that $G_k\geq k$. Hence, 
$$\lim_{n\to\infty}|\tilde A_n|=1+\sum_{k=1}^{\infty} {\bf 1}_{\{G_k\geq k\}}$$
is infinite almost surely. Hence, $(|\tilde A_n|)$ goes to infinity in probability as well and so does 
$(|A_n|)$.
\end{proof}

\subsection{The particular case of a uniform distribution}

Assume in this subsection that the random variable $C$ is uniformly distributed on $(0,a)$ for some $a$ in $(0,1)$. Hence, $E(1/C)=+\infty$ and $|A_n|$ goes to infinity in probability. We will now estimate the rate at which $|A_n|$ goes to infinity.
\begin{align*}
E(|A_n|)=&1+\sum_{k=0}^{n-1} P(G_k\geq n-k)\\
=&1+\sum_{k=0}^{n-1} E\left ((1-C)^{n-k-1}\right)\\  
=&1+\sum_{k=0}^{n-1}\frac{1}{a}\int_0^a (1-x)^{n-k-1}dx\\
=&1+\sum_{k=0}^{n-1}\frac{1}{(n-k)a}\left(1-(1-a)^{n-k}\right)\\
\end{align*}
Hence,
$$E(|A_n|)=1+\frac{1}{a}\sum_{k=1}^{n} \frac{1}{k}-\frac{1}{a}\sum_{k=1}^{n}\frac{(1-a)^{k}}{k}.$$
Therefore, as $n$ goes to infinity $E(|A_n|)$ grows as $\frac{1}{a}\ln n$, see Figure 1. Moreover, we have the following result,
\begin{thm}
 Assume that the random variable $C$ is uniformly distributed on $(0,a)$ for some $a$ in $(0,1)$. Then,
 $\frac{|A_n|}{a\ln n}$ converges in probability to 1.
\end{thm}
\begin{proof}
    By Chebyshev's inequality, for any $\epsilon>0$,
    \begin{equation}
     P\left(\left |\frac{|A_n|}{a\ln n}-\frac{E(|A_n|)}{a\ln n}\right|>\epsilon\right)\leq \frac{Var(|A_n|)}{a^2\ln^2 n} 
    \end{equation}
    
    Since $|A_n|$ is a sum of independent Bernoulli random variables ${\bf 1}_{\{G_k\geq n-k\}}$ for $0\leq k\leq n-1$, the variance of $|A_n|$ is the sum of the variances of these Bernoulli random variables. Note that
    $$E\left({\bf 1}_{\{G_k\geq n-k\}}\right)=\frac{1}{(n-k)a}\left(1-(1-a)^{n-k}\right).$$
    Thus, $Var(|A_n|)$ is of order $\frac{1}{a}\ln n$. Hence, by letting $n$ go to infinity in (1) we obtain that
    $$\left |\frac{|A_n|}{a\ln n}-\frac{E(|A_n|)}{a\ln n}\right|$$
    converges in probability to 0. Since $\frac{E(|A_n|)}{a\ln n}$ converges to 1 the result follows.
\end{proof}

\section{Related models.}

The present is a variation on the so-called catastrophe models that go back to at least the 1980's, see for instance \cite{Brockwell} and \cite{Brockwell et al.}. Recently, work has been done on stationary distributions in the case when $c$ is a constant, see \cite{Ben-Ari} and \cite{GH}. See also \cite{BS} for a model where $c$ is a function of the population size.

Closely related to the present model is the following model for population size which was introduced in \cite{FMS}. Consider a sequence of independent identically distributed random vectors, $(c_1,Z_1),(c_2,Z_2),\dots$ The $Z$ distribution is discrete with support on the set of natural numbers $\N$.
The $c$ distribution is continuous with support on $(0,1)$.
The size $X_n$ of the population at time $n\geq 0$ evolves as follows. Set $X_0=1$ and $B_0=0$. For $n\geq 1$, let
$$X_n=B_{n-1}+Z_n,$$
where the conditional distribution of $B_{n-1}$ given $(c_1,Z_1),\dots, (c_{n},Z_{n}),$ is distributed according to a binomial distribution with parameters $X_n$ and $1-c_{n}$. 

It is proved in \cite{FMS} that the Markov chain $(X_n)$ is positive recurrent if and only if $E(\ln Z)<+\infty$. A similar result is proved in \cite{Neuts} for a closely related model. As a consequence of this result we make the following two observations, see also Figure 2.

$\bullet$ In the case of constant death probability and $Z\equiv 1$ the Markov chain $(X_n)$ is positive recurrent.  

$\bullet$ Random environment and $Z\equiv 1$. Assume that the sequence $(c_n)$ of independent random variables is sampled from a fixed probability distribution on $(0,1)$. If we take the same death probability $c_n$ for all individuals at time $n$ then the Markov chain $(X_n)$ is positive recurrent.

Recurrence for a population means that its size will be down to one individual infinitely often. In fact, were it not for the reflecting barrier at 1 the population would get extinct with probability one. Thus, the two particular cases above seem to show that it is really the randomness of each individual $c$ that makes $(|A_n|)$ go to infinity in the present model (under $E(1/C)=+\infty$). 
Genetic diversity does help survival!

\section{The point process limit}

The random set
\begin{equation*}
  \Tilde A_\infty=\{(k,c_k): k\geq 0 \mbox{ and } G_k\geq k\}, 
  \end{equation*}
is the limit of the point process $(\tilde A_n)$ in the following sense.
We represent $(\tilde A_n)$ as a point measure as follows:
$$\tilde A_n \stackrel{\rm d}{=}\sum_{k=0}^n {\bf 1}_{\{G_k\geq k\}}\delta_{k,c_k},$$
where $\delta_{k,c_k}$ is a point mass on $(k,c_k)$. Hence, $\tilde A_n$ converges in distribution to 
$$\tilde A_\infty 
{=}\sum_{k=0}^\infty {\bf 1}_{\{G_k\geq k\}}\delta_{k,c_k}.$$

In the next two subsections we study properties of $\Tilde A_\infty$.

\subsection{An accumulation point}

The following result shows that individuals who survive long enough have corresponding $c$'s that accumulate near 0.

\begin{thm} 
\label{thm:acc}
Under the assumption $E(1/C)=+\infty$ the random set of points $\tilde A_\infty$ has a unique accumulation point at 0.  
\end{thm}

We have the following consequence of this result. For all $n\geq 0$, $A_n$ has the same distribution as $\tilde A_n$. Hence, $(A_n)$ converges in distribution to $\tilde A_\infty$. Thus, the limiting 
distribution of $(A_n)$ exists and has a unique accumulation point at 0.

\begin{proof}
We start with an auxiliary result.
\begin{lem}
\label{lem:acc}
    Let $0<b<1$. Then, almost surely there exists a natural number $n$ such that if $i\geq n$ and $G_i\geq i$ then $c_i<b$.
\end{lem}

Under $E(1/C)=+\infty$, the random set $\Tilde A_\infty$ has almost surely infinitely many points.
By Lemma \ref{lem:acc}, for any $0<b<1$ there are only finitely many points in $\tilde A_\infty$ which are above $b$. Therefore, there are infinitely many points below $b$. Thus, there is a unique accumulation point for the $c's$ at 0. This proves Theorem \ref{thm:acc}.

We now prove Lemma \ref{lem:acc}. Let $f$ be the probability density function of the random variable $C$. Then,
\begin{align*}
P(G_i\geq i, c_i>b)=&\int_b^1 (1-x)^{i-1}f(x)dx
\end{align*}

Thus,
\begin{align*}
\sum_{i\geq 1} P(G_i\geq i, c_i>b)=&\sum_{i\geq 1}\int_b^1 (1-x)^{i-1}f(x)dx\\
=&\int_b^1 \frac{1}{x}f(x)dx\\
\leq &\frac{1}{b}\\
<&+\infty
\end{align*}

Therefore, by the Borel-Cantelli Lemma there exists a.s. $n$ such that if $i\geq n$ either $G_i<i$ or $c_i<b$. The result follows.

\end{proof}

\subsection{ Convergence to a Poisson process}
Next we show that, for Uniform and Power Law type $C$ distributions the limit process properly rescaled converges to a non-homogeneous Poisson process.
\begin{thm}
Assume that $C$ is either uniformly distributed in $(0,1)$ or that the probability density function of $C$ is $f(x)=(1-\alpha)x^{-\alpha}$ on $(0,1)$ where $0<\alpha<1$. Then, the random set of points $\Tilde A_\infty$ properly rescaled converges in distribution to either a single Poisson process on $[0,+\infty)^2$, in the former case, or to a collection of independent Poisson processes on $[0,+\infty)\times{\mathbb R}$, in the latter case.
\end{thm}

\begin{proof} We will prove the result for $C$ uniformly distributed. Then we will explain how to modify the argument to prove the result when $C$ has probability density $f(x)=(1-\alpha)x^{-\alpha}$.\\
\hfill\\
\noindent
{\bf Uniform density.} Let $n\geq 1$ and $m\geq 1$ and consider two collections of disjoint finite intervals $J_k=(a_k,b_k)$ for $k=1,\dots,n$ and $I_j$ for $j=1,\dots,m$ in the positive real numbers.
Let $L>0$ be a fixed natural number and define for every $j$ and $k$, $I_j^{(L)}=LI_j$ and $J_k^{(L)}=\frac{1}{L}J_k.$
For every $(j,k)$, let $N^{(L)}(j,k)$ be the number of points of $\Tilde A_\infty$ that belong to the set $I_j^{(L)}\times J_k^{(L)}$.
That is,
$$N^{(L)}(j,k)=\sum_{\ell \in I_j^{(L)}}F^{(L)}(k,\ell),$$
where $F^{(L)}(k,\ell)$ is the indicator function of the set
$\{G_\ell>\ell,c_\ell\in J_k^{(L)}\}.$
We will compute the limit of the following Laplace transform. Let  ${\bf t}=(t(j,k))_{j,k}$ be an array of positive real numbers. Define the Laplace transform $M^{(L)}$ by,
$$M^{(L)}({\bf t})={\mathbb E}\left[\exp\left(-\sum_{j,k}t(j,k)N^{(L)}(j,k)\right)\right].$$
Then,
\begin{align*}
M^{(L)}({\bf t})=&
{\mathbb E}\left[\exp\left(-\sum_{j,k}t(j,k)\sum_{\ell \in I_j^{(L)}}F^{(L)}(k,\ell)\right)\right]\\
=&{\mathbb E}\left[\prod_j\prod_{\ell\in I_j^{(L)}}\exp\left(-\sum_k t(j,k)F^{(L)}(k,\ell)\right)\right]\\
\end{align*}
Note now that for $\ell_1\not=\ell_2$, $F^{(L)}(k,\ell_1)$ and $F^{(L)}(k,\ell_2)$ are independent. Hence,
\begin{align*}
M^{(L)}({\bf t})=&\prod_j\prod_{\ell\in I_j^{(L)}}{\mathbb E}\left[\exp\left(-\sum_k t(j,k)F^{(L)}(k,\ell)\right)\right]\\
\end{align*}
Observe now that for all $(k,\ell)$, $F^{(L)}(k,\ell)$ converges to $0$ almost surely as $L$ goes to infinity. Hence by the Dominated Convergence Theorem,
\begin{equation}
\label{approx}
\lim_{L\to\infty}\left(1-{\mathbb E}\left[\exp\left(-\sum_k t(j,k)F^{(L)}(k,\ell)\right)\right]\right)=0. 
\end{equation}
Thus,
$$\ln\left(1-\left(1-{\mathbb E}\left[\exp\left(-\sum_k t(j,k)F^{(L)}(k,\ell)\right)\right)\right]\right)$$
can be approximated by
$$ -\left(1-{\mathbb E}\left[\exp\left(-\sum_k t(j,k)F^{(L)}(k,\ell)\right)\right]\right),$$
as $L$ goes to infinity.
Therefore,

\begin{align*}
\ln M^{(L)}({\bf t})=&\sum_j\sum_{\ell\in I_j^{(L)}}\ln\left({\mathbb E}\left[\exp\left(-\sum_k t(j,k)F^{(L)}(k,\ell)\right)\right]\right)\\
=&\sum_j\sum_{\ell\in I_j^{(L)}}\ln\left(1-{\mathbb E}\left[1-\exp\left(-\sum_k t(j,k)F^{(L)}(k,\ell)\right)\right]\right)\\
\sim &-\sum_j\sum_{\ell\in I_j^{(L)}}{\mathbb E}\left[1-\exp\left(-\sum_k t(j,k)F^{(L)}(k,\ell)\right)\right]\\
\end{align*}
Note that for fixed $\ell$, the indicators $(F^{(L)}(k,\ell))_k$ are all $0$ except possibly for one $k$. Hence,
$$1-\exp\left(-\sum_k t(j,k)F^{(L)}(k,\ell)\right)=\sum_k \left(1-\exp(-t(j,k)\right)F^{(L)}(k,\ell).$$
Thus, as $L$ goes to infinity
\begin{equation}
\label{loglaplace}
\ln M^{(L)}({\bf t})\sim \sum_{j,k} \left(1-\exp(-t(j,k)\right)\sum_{\ell\in I_j^{(L)}}{\mathbb E}\left(F^{(L)}(k,\ell)\right).
\end{equation}
Note that the proof up to equation (\ref{loglaplace}) is valid for any probability density $f$ provided the limit in equation (\ref{approx}) holds true. Recall that we are presently dealing with the case $f(x)=1$ for $0<x<1$. Thus, for fixed $j$ and $k$,
\begin{align*}
\sum_{\ell\in I_j^{(L)}}{\mathbb E}\left(F^{(L)}(k,\ell)\right)=&\sum_{\ell\in I_j^{(L)}}\int_{x\in J_k^{(L)}}(1-x)^{\ell}f(x)dx\\
\end{align*}
Let $I_j^{(L)}=L[w_j,z_j]$ where $w_j$ and $z_j$ are positive integers. Recall that $J_k^{(L)}=\frac{1}{L}(a_k,b_k)$. Therefore,
\begin{align*}
\sum_{\ell\in I_j^{(L)}}{\mathbb E}\left(F^{(L)}(k,\ell)\right)=&\int_{a_k/L}^{b_k/L}\frac{1}{x}\left((1-x)^{Lw_j}-(1-x)^{Lz_j+1}\right)dx\\
=&\int_{a_k}^{b_k}\frac{1}{y}\left((1-\frac{y}{L})^{Lw_j}-(1-\frac{y}{L})^{Lz_j+1}\right)dy\\
\end{align*}
Observe that,
$$\left|(1-\frac{y}{L})^{Lw_j}-(1-\frac{y}{L})^{Lz_j+1}\right|\leq \exp(-yw_j)+\exp(-yz_j)
$$
Hence, by the Dominated Convergence Theorem,
\begin{align*}
\lim_{L\to\infty}\sum_{\ell\in I_j^{(L)}}{\mathbb E}\left(F^{(L)}(k,\ell)\right)=&
\int_{a_k}^{b_k}\frac{1}{y}\left(\exp(-yw_j)-\exp(-yz_j)\right)dy\\
=&\int_{a_k}^{b_k}\frac{1}{y}\int_{w_j}^{z_j} y\exp(-sy)dsdy\\
=&\int_{a_k}^{b_k}\int_{w_j}^{z_j}\exp(-sy)dsdy.
\end{align*}
Then,
$$\lim_{L\to\infty}\ln M^{(L)}({\bf t})=\sum_{j,k} \left(1-\exp(-t(j,k)\right)
\int_{a_k}^{b_k}\int_{w_j}^{z_j}\exp(-sy)dsdy.$$
This proves the following convergence in distribution.  Writing 
$$
\tilde A^{(L)}_\infty(B) = \tilde A_\infty([B]_L) 
$$
for any given Borel set $B$ of $[0,+\infty)^2$, where $L>0$ is a scaling parameter,
and $[B]_L=\{(Ls,\frac1Ly):(s,y)\in B\}$,
we have that, as $L\to\infty$, $\tilde A^{(L)}_\infty$ converges in distribution
to a non homogeneous Poisson process on $[0,+\infty)^2$ with intensity measure $\exp(-sy)dsdy$. 

\hfill\\
\noindent 
{\bf Power law density.} Assume that for $x\in (0,1)$, $f(x)=(1-\alpha)x^{-\alpha}$ for a fixed $\alpha$ in $(0,1)$. This proof parallels the preceding one with a different rescaling that we now define. Let $L>0$ be a fixed natural number. Fix an arbitrary $m\geq1$; for $j=1,\ldots,m$, 
let $I_j^{(L)}=L[w_j,z_j]$ where $w_j$ and $z_j$ are positive integers
such that $[w_j,z_j]$, $j=1,\ldots,m$ are disjoint intervals. Now, fix $n\geq1$ and choose $a_1,\ldots,a_n$ distinct positive numbers, and for each $i,k=1,\ldots,n$, define finite intervals $J_{ik}=(b^i_k,d^i_k)$ in the real numbers, which are disjoint in $k$ for every $i$. For $i,k=1,\dots,n$, let
$$J_{ik}^{(L)}=(\frac{1}{L}a_i+\frac{1}{L^\beta}b^i_k,\frac{1}{L}a_i+\frac{1}{L^\beta}d^i_k),$$
where $\beta=1/(1-\alpha)$. Notice that $\beta>1$. The indicator functions $F^{(L)}(i,k;\ell)$ are defined as above (with $J_{ik}^{(L)}$ replacing $J_k^{(L)}$). It is easy to see that the limit in (\ref{approx}) holds true. Hence, the approximation leading to (\ref{loglaplace}) is still valid.
Then, a computation similar to the one above shows that for all $j$ and $i,k$
\begin{align*}
\lim_{L\to\infty}\sum_{\ell\in I_j^{(L)}}{\mathbb E}\left(F^{(L)}(i,k;\ell)\right)=&
\int_{b^i_k}^{d^i_k}\int_{w_j}^{z_j}\frac{1-\alpha}{a_i^\alpha}\exp(-a_i s)dsdy.
\end{align*}
Therefore, writing 
$$\tilde A_\infty^{(a,L)}(B)=\tilde A_\infty([\![ B ]\!]^a_L)$$ 
for any given Borel set $B$ of $[0,+\infty)\times{\mathbb R}$, where $a>0$ is a location parameter, $L>0$ is a scaling parameter, and $[\![ B ]\!]^a_L=\{(Ls,\frac1La+\frac1{L^\beta}y):(s,y)\in B\}$,
we have that, as $L\to\infty$, $\{\tilde A^{(a,L)}_\infty,a>0\}$ converges in distribution to a collection of independent non homogeneous Poisson process on $[0,+\infty)\times{\mathbb R}$, with marginal intensity measure  $\frac{1-\alpha}{a^\alpha}\exp(-s a)dsdy$.
\end{proof}

\bibliographystyle{amsplain}

\begin{thebibliography}{100}


\bibitem{Ben-Ari}
I.Ben-Ari, A. Roitershtein and R.B. Schinazi (2019).
A random walk with catastrophes. Electronic Journal of Probability 24, 1-21.

\bibitem{BS}
I.Ben-Ari, and R.B. Schinazi (2023). Can a single migrant per generation rescue a dying population?
arXiv:2304.06478

\bibitem{Brockwell}
P.~J.~Brockwell (1986).
The extinction time of a general birth and death process with catastrophes.
J. Appl. Probab. 23, 851-858.

\bibitem{Brockwell et al.}
P.~J.~Brockwell, J.~Gani and S.~I.~Resnick (1982).
Birth, immigration and catastrophe processes.
Adv. in Appl. Probab. 14, 709-731.

\bibitem{FMS}
L.R.Fontes, F. Machado and R.B. Schinazi (2023) Null recurrence and transience for a binomial catastrophe model in random environment. J. Stat. Mech. 033201

\bibitem{GH}
B. Goncalves and T. Huillet  (2021)
A generating function approach to Markov chains undergoing binomial catastrophes.
J. Stat. Mech. 033402


\bibitem{Neuts}
M.~F.~Neuts (1994). An interesting random walk on the non-negative integers.
J. Appl. Probab. 31, 48-58.


\end{thebibliography}

\begin{figure}[ht]
	\includegraphics[width=10cm]{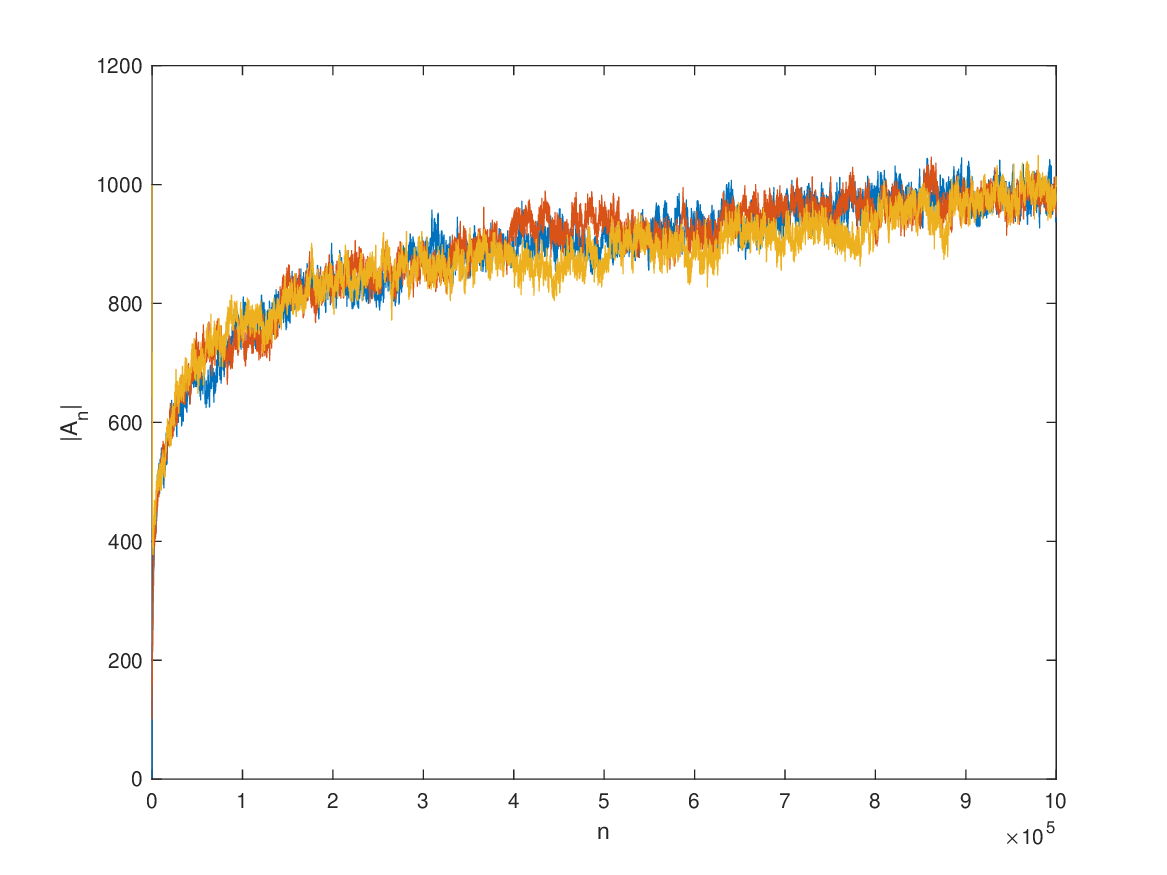}
 
\caption{These are 3 simulations of $(|A_n|)$ (i.e. the total size of the population $A_n$) starting with 1, 100 and 1000 individuals, respectively. The death probabilities are sampled from a uniform on $(0,a)$ with $a=0.01$. }
\end{figure}	

\vskip1cm 

\begin{figure}[ht]
	\includegraphics[width=8cm]{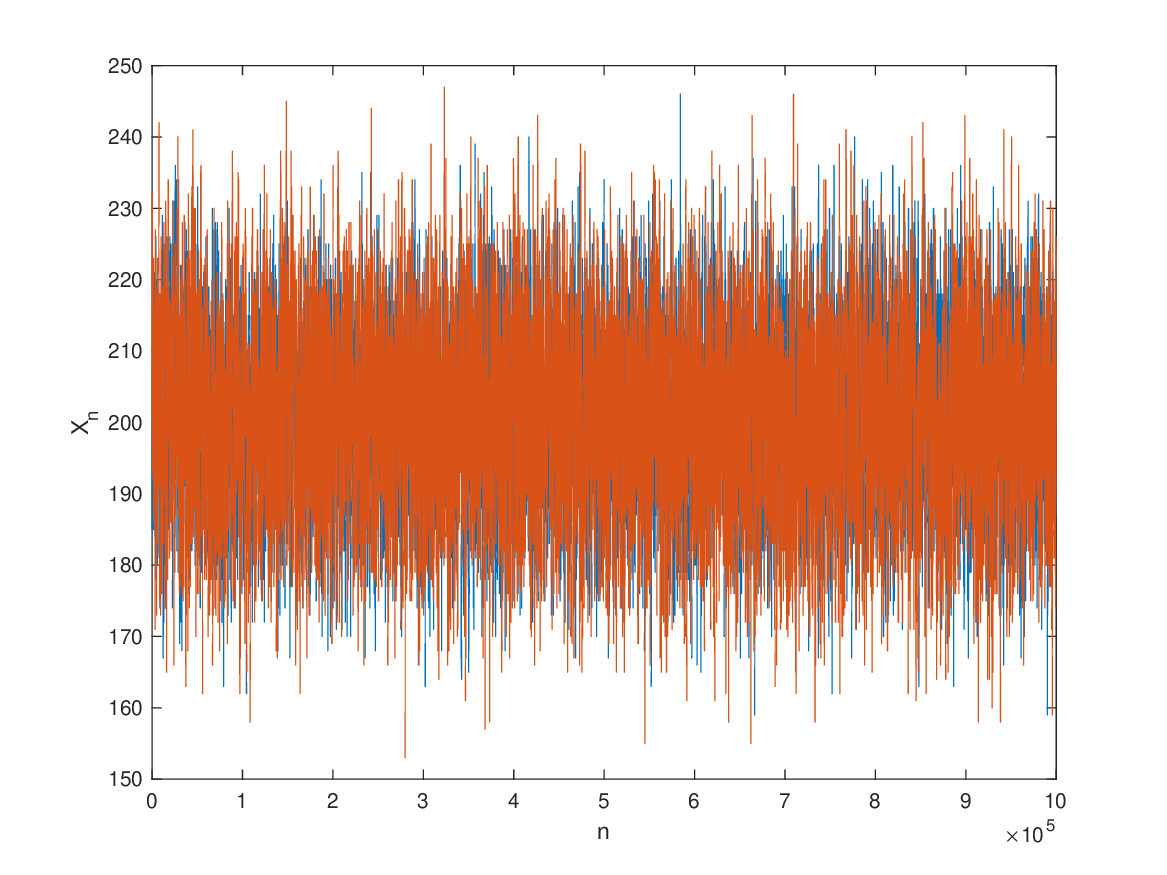}
\caption {These are simulations of the population size for two models from \cite{FMS}.  One model has constant $c=0.005$. The other model has a random environment with $c$ sampled from a uniform on $(0,a)$ with $a=0.01$. Both models are positive recurrent and their paths fluctuate in the same narrow strip. This in sharp contrast with what we see in Figure 1.}
	\end{figure}

\end{document}